\documentclass[12pt]{article} 

\usepackage[utf8]{inputenc} 

\usepackage{amssymb}
\usepackage{amsfonts}
\usepackage{amsmath}
\usepackage{amsxtra}
\usepackage{amsthm}
\usepackage{mathrsfs}
\usepackage{enumerate}
\usepackage{txfonts} 

\usepackage{geometry} 
\usepackage{graphicx} 

\usepackage{booktabs} 
\usepackage{array} 
\usepackage{paralist} 
\usepackage{verbatim} 
\usepackage{subfig} 

\usepackage{fancyhdr} 
\newtheorem{thm}{Theorem}[section]
\newtheorem{lem}[thm]{Lemma}
\newtheorem{cor}[thm]{Corollary}
\newtheorem{prop}[thm]{Proposition}
\renewcommand{\epsilon}{\varepsilon}
\title{Optimal regularity and nondegeneracy of a free boundary problem related to the fractional Laplacian}
\author{Ray Yang\\New York University\\ \texttt{ryang@cims.nyu.edu}}
\date{} 
\begin{document}
\maketitle
\begin{abstract}
We discuss the optimal regularity and nondegeneracy of a free boundary problem related to the fractional Laplacian. This work is related to, but addresses a different problem from, recent work of Caffarelli, Roquejoffre, and Sire \cite{C-R-S}. A variant of the boundary Harnack inequality is also proved, where it is no longer required that the function be 0 along the boundary. 
\end{abstract}
\section{Introduction}
In this paper, we prove certain local properties near the free boundary for minimizers of the following energy functional,
\[ J(u) = \frac{1}{2} \int y^{1-2\sigma} (\nabla u)^2 dx dy + \int_{\{y=0\}} u^\gamma dx, \]
where $(x,y)$ lies in the upper half space $\mathbb{R}^n \times \mathbb{R_+}$, and $0 < \sigma, \gamma < 1$, subject to $u \geq 0$. We are interested in the properties of $u$ along the hyperplane $\{y=0\}$. The first term of the energy functional is related to the fractional Laplace operator, and the second is thought of as imposing an energy penalty when $u>0$. When the local values of $u$ are sufficiently small (say, on $\partial B_1 \cap \{y>0\}$), the set $\{u=0\}$ is nontrivial, lying on $\{y=0\}$. The boundary of that set in the topology of $\mathbb{R}^n$ (that is, the $x$ variable only) is called the free boundary. 

We prove that in any neighborhood of a free boundary point, our minimizer is bounded by a power of the distance to the free boundary; specifically, if 0 is a free boundary point, then $\sup_{x\in B_r} u(x,0) \leq C r^\beta$. When combined with interior estimates, this gives the $C^\beta$ regularity of energy minimizers with a Holder seminorm that depends only on the distance to the free boundary. The exponent $\beta = \frac{2\sigma}{2-\gamma} $ is the critical scaling exponent for the problem. This is called the \textit{optimal regularity} of $u$, since we also prove the \textit{non-degenerate} nature of $u$, namely, that in any ball of radius $r$ about a free boundary point, $\sup_{x\in B_r} u(x,0) \geq C r^\beta$ for a constant $C$ that depends only on $n,\sigma,$ and $\gamma$. In the course of proving the optimal regularity, we prove and use an improvement on the boundary Harnack inequality of Caffarelli, Fabes, Mortola, and Salsa \cite{C-F-M-S}, which may be of interest even to those not working in free boundaries. 

The motivation for the problem comes from the intersection of the study of nonlocal integrodifferential operators like the fractional Laplacian and the study of one-phase free boundary problems. 

In the theory of one-phase free boundary problems arising from the minimization of energy for the \textit{classical} Laplacian, Alt and Caffarelli \cite{A-C} analyzed minimizers of the energy $J(u) = \frac{1}{2} \int (\nabla u)^2 + \chi_{u > 0} dx$ subject to non-negative Dirichlet data, while the study of the free boundary arising from minimizers of the energy $J(u) = \int \frac{(\nabla u)^2}{2} + u dx$ with non-negative Dirichlet data is encompassed by the study of the obstacle problem. An intermediate case is the case studied by Alt and Phillips \cite{A-P}, which is that of the free boundary for minimizers of $J(u) = \int \frac{(\nabla u)^2}{2} + u^\gamma dx$, where $0<\gamma<1$. In a heuristic sense, we can view the Alt-Caffarelli problem as the case of $\gamma = 0$, and the case $\gamma =1$ as a special case of the obstacle problem. The fractional case of the Alt-Caffarelli problem was recently analyzed in \cite{C-R-S}; that paper was the direct inspiration for this one. 

The problem we study is the analogue of the problem of Alt and Phillips for the fractional, rather than standard, Laplace operator. The current article only covers the regularity and nondegeneracy of energy minimizers, and is thus properly the analogue of Phillips' work in \cite{P-1} and part of \cite{P-2}. The results in \cite{P-1} were extended by Giaquinta and Giusti to the two phase case \cite{Giaquinta-Giusti}. 

The fractional Laplace is a nonlocal integral operator, taking the form 
\[ (-\Delta)^\sigma u(x) = C_{n,\sigma} \int_{\mathbb{R}^n} \frac{u(x)-u(z)}{|x-z|^{n+2\sigma}} dz.\] 
This operator has a corresponding energy given by
\[ J(u) = \int_{\mathbb{R}^n \times \mathbb{R}^n} \frac{(u(x)-u(z))^2}{|x-z|^{n+2\sigma}} dz dx \]
This latter term is a nonlocal energy, and not very easy to manipulate. In \cite{C-S}, Caffarelli and Silvestre introduced the notion of extension to one extra spatial dimension and examining a particular PDE on the upper half-space, with the fractional Laplacian being equivalent to the Dirichlet-to-Neumann map at the boundary. To be precise, consider a function $u \in H^s(\mathbb{R}^n)$. Denoting the coordinates $(x,y) \in \mathbb{R}^n \times \mathbb{R}_+$, such a function can be extended by means of a suitable Poisson kernel to a function $u(x,y)$, lying in $H^1(a,\mathbb{R}^n \times \mathbb{R}_+$, where the energy seminorm is given by $[u]_{H^1(a)} = \int_{\mathbb{R}^{n+1}_+} y^a |\nabla u|^2 dx dy$, where $u(x,0) = u(x)$. This function will satisfy the equation
\[ \nabla \cdot (y^{1-2\sigma} \nabla u) = 0 \]
Then we have
\[ (-\Delta)^\sigma u(x) =  - C_{n,\sigma} \lim_{y\rightarrow 0} y^{1-2\sigma} \partial_y u(x,y) \]
The natural energy corresponding to the second-order equation on the half-space is then $J(u) = \int y^{1-2\sigma} |\nabla u (x,y)|^2 dx dy$. This is an energy where it is easier to study the purely local properties of its minimizers. 

The extension characterization of the fractional Laplacian has been used to study both the obstacle problem (\cite{C-S-S}) and the $\gamma = 0$ case (\cite{C-R-S}). We will apply it to study the intermediary case, that is to say, minimizers of our energy functional 
\[ J(u) = \int y^{1-2\sigma} |\nabla u(x,y)|^2 dx dy + \int_{\{y=0\}} u^\gamma dx \]
in subsets of the upper half-space with parts of their boundary lying along $y = 0$, where we have $0 < \gamma < 1$. Since this is a study of the one-phase problem, we assume non-negative Dirichlet boundary conditions. The second term in the energy penalizes non-zero values of $u$ along the hyperplane $\{y=0\}$. Hence, we can consider separately the zero set of $u$ (called the \textit{contact set}), and its positivity set. Restricted to $\{y=0\}$, the interface between the two is the \textit{free boundary}. The optimal regularity as $\gamma \rightarrow 0$ is indeed $C^\sigma$, which accords with the regularity when $\gamma = 0$ (\cite{C-R-S}), but the regularity for the fractional obstacle problem is $C^{1,\sigma}$ (\cite{C-S-S}), which is not the limit as $\gamma \rightarrow 1$ of the regularity found here. This jump in regularity is interesting. 

The ideas behind the proof of the optimal regularity of energy minimizers for fractional-order cases can be extended to a proof of optimal regularity for the second-order case, which was first proved by Phillips \cite{P-1}. Since the proof for the second-order case illustrates the ideas in a less involved setting than the fractional-order case, we provide it as well. The key ingredient for the proof is the construction of a lower barrier, or subsolution, for the energy minimizer which is strictly positive at the center of a ball when the values near-by are ``too large,'' thus, for a free boundary point to exist, the growth cannot be too great. 

\section{Preliminary considerations}
In this section we identify some technical points of interest. First, we prove that minimizers of the energy exist. Second, we identify the scaling associated with the problem. Third, we list certain properties of the equation and minimizer that are known and will prove useful to our analysis.

\subsection{Existence considerations and some definitions}
We consider, in $B_+ = B_1 \cap \{y \geq 0\}$, minimizers of the energy
\[ J(u) = \frac{1}{2}\int_{B_+} y^a |\nabla u|^2 dx dy + \int_{B_1 \cap \{y=0\}} u^\gamma dx \]
in the space $H^1(B_+,a)$, with seminorm given by $\|u\| = \int_{B_+} y^a (|\nabla u|^2) dx dy$. We impose non-negative Dirichlet conditions on $\partial B_1 \cap \{y > 0\}$, where $a = 1-2\sigma$, $0<\sigma<1$, $0<\gamma<1$. For sake of convenience, we denote $\Gamma = B_1 \cap \{y=0\}$. 

The energy can be interpreted as an averaging term which ``lifts'' the solution towards the boundary conditions, and a term which punishes $u$ for being nonzero at $y=0$, causing it to ``stick.'' The set $\{u=0\}$, which necessarily lies in $\{y=0\}$, is called the \textit{contact set} of $u$. The interface between $\{u=0\} \cap \{y=0\}$ and $\{u>0\} \cap \{y=0\}$ is called the \textit{free boundary}. 

Existence of minimizers is assured by the usual methods: consider a minimizing sequence for the energy. The first term of the energy is lower semicontinuous with respect to the norm for the usual reasons. The second term is continuous with respect to the norm for $L^2(\Gamma)$ From the extension result of Caffarelli-Silvestre we know that the trace of functions lying in $H^1(B_+,a)$ lie in $H^\sigma(\Gamma)$ \cite{C-S}, whence we apply the usual Sobolev embedding of $H^\sigma$ inside $L^2$. 

We will use $X = (x,y)$, where $x \in \mathbb{R}^n$ and $y \in \mathbb{R}_+$. 

\subsection{Scaling of the problem}\label{scaling-section}
Put briefly, we seek the scaling that preserves minimizers of the energy 
\[ J(u) = \frac{1}{2} \int_{B_+} y^a |\nabla u|^2 dx dy + \int_\Gamma u^\gamma dx \]
We consider the scaling
\[ w(x,y) = \frac{1}{\lambda^\beta} u(\lambda x, \lambda y) \]
We find that 
\[ J(w) = \frac{1}{2} \int_\Omega y^a \lambda^{2-2\beta} |\nabla u(\lambda x, \lambda y)|^2 dx dy + \int_\Gamma \lambda^{-\beta \gamma} u^\gamma dx \]
which, after the change of variable, scales to
\[ \lambda^{-a-n+2-2\beta} \frac{1}{2} \int_{\lambda \Omega}  y^a |\nabla u|^2 dx dy + \lambda^{-\beta \gamma -n + 1} \int_{\lambda \Gamma} u^\gamma dx \]
Setting the exponents equal, we find that
\[ \beta = \frac{2\sigma}{2-\gamma} \]

\subsection{Some other useful properties}
The Euler-Lagrange equations for $J(u)$ tell us that, in a distributional sense, the minimizer $u$ satisfies
\[ \nabla \cdot (y^a \nabla u) = 0 \]
in the interior of $B_+$, and 
\[ \lim_{y\rightarrow 0} y^a \partial_y u = \gamma u^{\gamma -1 } \]
along $\Gamma$ wherever $u > 0$. Caffarelli and Silvestre \cite{C-S} showed that, in the upper half space $\mathbb{R}_+^{n+1}$, the effective normal derivative operator is equivalent to the fractional Laplacian of order $\sigma$:
$$\lim_{y\rightarrow 0} y^a \partial_y u = -C (-\Delta)^\sigma u $$

Fabes, Jerison, Kenig, and Serapioni (\cite{F-K-S},\cite{F-K-J-1},\cite{F-K-J-2}) extended the De Giorgi-Nash-Moser theory of divergence-form elliptic equations to degenerate elliptic equations with Muckehnhoupt $A_2$ weights; these are equations of the form 
\[ \nabla \cdot (A(x) \nabla u) = 0 \]
where the matrix $A$ satisfies
\[ \lambda w(x) |\xi|^2 \leq \sum A_{ij}(x) \xi_i \xi_j \leq \Lambda w(x) |\xi|^2 \]
for $\xi \in \mathbb{R}^n$, with weight functions $w$ satisfying
\[ \left(\frac{1}{|B|} \int_B w(x) dx \right) \left(\frac{1}{|B|} \int \frac{1}{w(x)} dx \right) \leq C \]
for all balls $B$. In particular, such properties as the strong maximum principle, Holder regularity, and the Harnack inequality all hold. There are certain other properties, such as the De Giorgi Oscillation Lemma and a more specific form of the De Giorgi-Nash-Moser Harnack inequality, which follow directly from their work but are not explicitly stated. A discussion of those properties has been provided in \S \ref{ap:degiorgi} for the reader's convenience. 

Finally, the generalized Hopf lemma for $\sigma$-harmonic functions in $\mathbb{R}^n$ (stated in \cite{C-R-S}) is useful:
\begin{lem}\label{lem:Hopf}
If a smooth function $v(x)$ satisfies $(-\Delta)^\sigma v = 0$ in some smooth domain $\Omega \subset \mathbb{R}^n$, $v$ being non-negative and positive in the interior of $\Omega$, and if there exists a point $x_0 \in \partial_\Omega$ where $v(x) = 0$, then there exists $C$ such that $v(x) \geq C ( (x-x_0) \cdot \nu(x_0))^\sigma $ where $\nu(x_0)$ is the interior normal to $\partial \Omega$ at $x_0$. 
\end{lem}

\section{Optimal regularity for the 2nd order case}
The optimal regularity for the problem in the 2nd order case was first obtained by Phillips \cite{P-1}. Our method for proving the optimal regularity of the fractional case can be adapted to give an alternative proof for Phillips' main thereom (Theorem II in \cite{P-1}). The main intuition behind our proof in the 2nd order case is free of certain technical issues that occur in the fractional case, and so we present it here first. 

We obtain the optimal regularity of the energy minimizer $u$ to the energy functional
\[ J(u) = \int \frac{|\nabla u|^2}{2} + u^\gamma dx, \]
showing that $u \in C^{1,\beta-1}$ where $\beta = \frac{2}{2-\gamma}$, the scaling factor obtained by a calculation like that in \S \ref{scaling-section}. As with the fractional case, we assume the boundary data is non-negative, which allows us to assume the same for $u$. Notice that the Euler-Lagrange equations tells us that, when $u > 0$, $u$ satisfies 
\[ \Delta u = \gamma u^{\gamma -1} \]
As with Phillips, we seek to prove
\begin{lem}\label{lem:2nd-order-regularity}
There exists a constant $c_0(n)$ such that if 
\[\fint_{B_r} u dS > c_0 r^\beta \]
then $u > 0$ inside $B_r$. 
\end{lem}
Since the scaling $u_\lambda(x) = \frac{1}{\lambda^\beta} u(\lambda x)$ preserves minimizers, we need only show this for $r=1$, and scaling would take care of the rest.

Our proof works by showing that, when the average on the boundary is sufficiently large, a subsolution, or lower barrier to the energy minimizer, can be constructed which is wholly positive in the interior of $B_r$. There are two main stages to the proof: first, we detail what it means to be a subsolution, second, we construct a subsolution with the desired properties. 

\subsection{Subsolutions}
We say that a function $w$ is a comparison subsolution, or lower barrier, to the energy minimizer, if $w$ satisfies 
\[ \Delta w \geq 2 w^{\gamma -1} \] 
whenever $w>0$ inside $B_r$, setting $w = u$ along $\partial B_r$ for Dirichlet boundary condition. This terminology is natural because, as we shall see, $u \geq w$ inside $B_r$. We now justify this definition. 

Let $v = \max(u,w)$, and consider the difference of energies given by 
\[ J(u) - J(v) = \int \frac{1}{2} \left(( \nabla u) ^2 - (\nabla v)^2\right)  + u^\gamma - v^\gamma dx \]
Since $u$ is the energy minimizer, we require that $J(u) - J(v) \leq 0$. However, we know that 
\begin{eqnarray*}
\frac{1}{2} \int (\nabla u)^2 - (\nabla v)^2 dx &=& \frac{1}{2} \int \nabla (u+v) \cdot \nabla (u-v) dx \\
&=& \frac{1}{2} \int (v-u) \Delta (u + v) dx \\
&\geq& \frac{1}{2} \int_{v>u>0}  (v - u) (\gamma u^{\gamma -1} + 2 v^{\gamma -1}) dx + \frac{1}{2} \int_{v > u = 0} 2 v^\gamma dx 
\end{eqnarray*}
We compare 
\[ \psi(s) = s^\gamma - t^\gamma \]
with
\[ \phi(s) = \frac{1}{2} (s-t) (\gamma t^{\gamma -1} + 2 s^{\gamma -1}) \]
for $s > t \geq 0$.  Clearly, $\phi(t) = \psi(t) = 0$, and a bit of calculation assures us that $\phi'(s) \geq \psi'(s)$ for all $s$. Thus, 
\begin{eqnarray*}
J(u) - J(v) &=& \int \frac{1}{2} \left(( \nabla u) ^2 - (\nabla v)^2\right)  + u^\gamma - v^\gamma dx \\
&\geq& \frac{1}{2} \int_{v>u>0}  (v - u) (\gamma u^{\gamma -1} + M v^{\gamma -1}) dx + \frac{1}{2} \int_{v > u = 0} M v^\gamma dx  \\
&+& \int u^\gamma - v^\gamma dx \\
&\geq& 0
\end{eqnarray*}
with equality holding in the last statement only if $v \equiv u$. Thus, $u \geq w$ for any comparison subsolution $w$. 

\subsection{Construction of a positive subsolution: Proof of Lemma \ref{lem:2nd-order-regularity}}
Our goal is to create a positive function $w$ on the unit ball $B_1$, such that $\Delta w \geq 2 w^{\gamma -1}$. We will define $w$ in three parts: 
\[ w(x) = w_1(x) + w_2(x) + w_3(x) \]

Let $\eta(x)$ be a radial, non-negative $C^\infty$ function satisfying $\eta \leq 1$ everywhere on $B_1$, $\eta = 1$ when $1 \geq |x| > \frac{3}{4}$, $\eta = 0$ when $|x| < \frac{1}{2}$, and $|\Delta \eta| \leq C' $ and $|\nabla \eta| \leq C' $ for some constant $C'$. We define 
\[ w_1(x) = \lambda \left(\eta(x) (1-|x|)^\beta + (1-\eta(x))\right) \]
and we claim that when $1 \geq |x| > \frac{7}{8}$, we have 
\[ \Delta w_1 \geq 2 w_1^{\gamma -1} \]
for the correct choice of $\lambda$. 

In the region in question, it is easy to see that 
\[ \Delta w_1(x) = \lambda \left(\beta (\beta -1) (1-r)^{\beta -2} - (n-1) \beta \frac{(1-r)^{\beta -1}}{r}\right) \]
We take the ratio of $\Delta w_1$ with $w_1$ in the region under concern, and we see that
\[ \frac{\Delta w_1}{w_1^{\gamma -1}} = \beta \lambda^{2-\gamma} \left( (\beta-1) \left(\frac{1-r}{r}\right) \right) \]
whence it is clear that a sufficiently large value of $\lambda$ will suffice. 

We set 
\[ w_2(x) = \mu (|x|^2 -1) \]
where we pick $\mu$ sufficiently large so that $\Delta w_2 > -\Delta w_1 + 1$ everywhere inside $B_\frac{7}{8}$. It is clear from our design that $-\Delta w_1$ is bounded inside the region in question. 

Finally, we let $w_3$ be the function which is harmonic inside $B_1$, with the same boundary values as the minimizer $u$ along $\partial B_1$. 

We claim that when $\fint_{\partial B_1} u dS$ is sufficiently large
\begin{enumerate}
\item $1 \geq 2 w^{\gamma -1}$ on $B_{\frac{7}{8}}$. 
\item $w \geq 0$ everywhere on $B_1$ (in fact, $w_2 + w_3 \geq 0$, while $w_1 \geq 0$ by construction). 
\end{enumerate}
The Harnack inequality tells us that, on $B_{\frac{7}{8}}$, we have
\[ w_3(x) \geq C \fint_{\partial B_1} u dS \]
To prove the first, it suffices if $\fint_{\partial B_1} u dS$ is so large compared to $\mu$ so that $2 \mu n > (C \fint_{\partial B_1} u dS)^{\gamma -1}$. To prove the second, we bound $w_3$ from below by a suitably scaled truncated fundamental solution, and then
\[ w_3(x) \geq \frac{C \fint_{\partial B_1} u dS}{ (\frac{8}{7})^{n-2} - 1} \left( \frac{1}{|x|^{n-2}} - 1 \right) \geq \mu (1 - |x|^2) \]
when $\fint_{\partial B_1} u dS$ is sufficiently large. 

Hence, we have $w \geq 0$ everywhere, and on $B_\frac{7}{8}$, we have 
\[ \Delta w \geq 1 \geq 2 w^{\gamma -1} \]
while on $B_1 \setminus B_\frac{7}{8}$ we have
\[ \Delta w \geq \Delta w_1 \geq 2 w_1^{\gamma -1} \geq 2 w^{\gamma -1} \]
(the last inquality is because $a^t + b^t \geq (a+b)^t$ when $a,b \geq 0, t < 0$. 

\section{Optimal regularity for the fractional case}
The goal of this section is to obtain the optimal regularity of energy minimizers $u$. In particular, we seek to show that $u$ grows away from the free boundary like a power of the distance. To be precise, $u(X) \leq C d^\beta$, where $d$ is the distance of $X$ to the free boundary, and $\beta = \frac{2\sigma}{2-\gamma}$ is the scaling factor obtained in \S \ref{scaling-section}.

As corollaries, we obtain some regularity results: restricted to $\{y=0\}$, $u$ lies in the Holder space $C^\beta$. If $\beta > 1$, we will prove that $u \in C^{1,\beta-1}$, which by abuse of notation we will still refer to as $C^\beta$. In the interior domain where $y>0$, we still obtain $u\in C^\beta$ when $\beta < 1$. When $\beta \geq 1$, we find that $u \in C^\alpha$ for any $\alpha < 1$. 

To obtain optimal growth, we consider a point $p_0$ which is at some distance (normalized to 1) from the nearest free boundary point, which we will take to be 0. We will use a variant of the boundary Harnack inequality due to Caffarelli to compare the value of $u(p_0)$ with some point $p_1$ in the interior of $B_+$, specifically, showing that $u(p_1) \geq M u(p_0)$. We can then use the regular Harnack inequality in the interior to show that, in a smaller ball about the free boundary point, the boundary values are controlled by $u(p_1)$, and hence by $u(p_0)$. We then prove that, if the boundary values in the upper half ball are too large at the right scale, then $u(0)$ is strictly positive, meaning that $u(p_0)$ cannot be too large. Subsequently, rescaling obtains the desired regularity. 

Our main tools to prove optimal growth are a variant of the Boundary Harnack Inequality\footnote{There are two types of results which are, confusingly, both called the Boundary Harnack Inequality in the literature. In addition to the result here, which states that values in the neighborhood of the boundary are uniformly bounded in terms of the value at an interior point, there is a closely related result which states that for two solutions which are both 0 along a stretch of the boundary, their ratios are locally Holder-continuous. We follow the naming convention of Caffarelli and Salsa \cite{Caffarelli-Salsa} and call the first result the Boundary Harnack Inequality, and the second result the Boundary Comparison Principle.}, and a lemma stating that if a particular weighted integral along the boundary of the a half-sphere is sufficiently large, then the minimizer of the energy taking boundary conditions along the sphere has a positive value at the center. 

\begin{thm}[Variant Boundary Harnack Inequality]\label{variant-Harnack}
Let $u$ be a non-negative solution of the equation $\nabla \cdot (A \nabla u) = 0$ in $B^+$, where $A$ satisfies the Muckenhoupt $A_2$ condition, with $\lim_{y\rightarrow 0} A \nabla u \cdot \hat{y} \geq 0$ along $\{y=0\}$, taking on some continuous boundary values along $\{y=0\}$, with $u(0,\frac{1}{4}) = 1$. Then inside $B^+_{\frac{1}{2}}$, we have $u \leq M$ for some constant $M(n,\sigma)$. 
\end{thm}

\begin{lem}[Minimizers with large averages are positive at the center]\label{reg-lem}
Let $u$ be a minimizer of the energy $J(u)$ inside $B_r \cap \{y>0\}$, taking non-negative boundary values along $\partial B_r \cap \{y>0\}$. $\exists c_0 > 0$ such that, if
\[ u|_{\partial B \cap \{y > \frac{r}{2} \}} \geq c_0 \]
then we have
\[ u(x,y) > 0 \forall (x,y) \in B_{\frac{r}{3}} \cap \{y\geq0\} \]
and in fact there exists a constant $c$ such that
\[ u(x,y) \geq c \forall (x,y) \in B_{\frac{r}{6}} \cap \{y \geq 0\}\]
\end{lem}

Together, these suffice to prove our result, namely: 
\begin{thm}\label{thm:opt-reg}
There exists a constant $K$ such that in any $B_r(x_0) \cap \Gamma$, where $x_0$ is a point such that $u(x_0) = 0$, such that
\[ |u(x) - u(x_0)| \leq K |x - x_0|^\beta \]
where $\beta = \frac{2 \sigma}{2-\gamma}$. 
\end{thm}
\begin{proof}
Without loss of generality, let 0 be a point such that $u(0) = 0$, and $X_*$ be a point such that $|X_*| = 1$. We claim that $u(X_*) \leq K$. Suppose this is not true, that is to say, we can make $u(X_*)$ as large as we wish. Then by the variant boundary Harnack inequality applied to $B_2(0)$ we have 
\[ u(0,1) \geq \frac{u(X_*)}{M} \]
where $M$ is the constant from the variant boundary Harnack inequality. By applying the DeGiorgi-Nash-Moser Harnack inequality to $u$ in a region containing  
\[ u(x,y) \geq C u(X_*) \]
whence, by invoking Lemma \ref{reg-lem}, we have $u(0) > 0$, a contradiction on our original assumption. Thus, there exists a constant $K$ such that $u(X_*) \leq K$, as desired. By rescaling the problem, we recover our desired result. 
\end{proof}

\begin{cor}\label{thm:c-beta}
Let $u$ be an energy minimizer in a subset of $\mathbb{R}^{n+1}_+$ containing $B_1^+$, with 0 a free boundary point. Then considered as a function along the set $\{y=0\}$, $u$ is a $C^\beta$ function, with $\|u\|_{C^\beta(B_\frac{1}{2})} \leq C$, where $C$ depends only on $\sigma, \gamma$, and $n$. 
\end{cor}

\begin{cor}\label{thm:global-beta}
Let $u$ be an energy minimizer in a subset of $\mathbb{R}^{n+1}_+$ containing $B_1^+$ $B_1^+$, with 0 a free boundary point. Then in $B_\frac{1}{2}^+$, $u$ is a $C^\beta$ function, with $\|u\|_{C^\beta(B_\frac{1}{2}^+)} \leq C$, where $C$ depends only on $\sigma, \gamma$, and $n$, if $\beta < 1$. If $\beta \leq 1$, $u$ is a $C^\alpha$ function for any $\alpha < 1$, with 
\[ \|u\|_{C^\alpha(B_\frac{1}{2}^+} \leq C(\sigma, \gamma, n, \alpha) \]
\end{cor}

The proof of these statements, and a discussion of the $C^\beta$ norm estimates, are covered in \ref{subsec:c^beta-estimates}. 

\subsection{Variant Boundary Harnack Inequality: Proof of Theorem \label{variant-Harnack}}
The proof of this variant of the boundary Harnack inequality follows the same lines as the standard proof of the boundary Harnack inequality provided by Caffarelli et alia \cite{C-F-M-S}. The proof uses two classical facts from the De Giorgi-Nash-Moser theory, which was extended to the theory of degenerate elliptic equations with $A_2$ weights by Fabes, Kenig, and Serapioni \cite{F-K-S}, a class that includes the equation $\nabla \cdot (y^a \nabla v) = 0$. 

The first fact is the De Giorgi-Nash-Moser Harnack inequality, which states that for a non-negative solution in $B_1$, 
\[ \sup_{B_r} u \leq c (1-r)^{-p} \inf_{B_r} u \]
where $p>0$. 

The second fact is the De Giorgi oscillation lemma, which says that a subsolution $v$ in the unit ball, or, in our case, $B_1 \cap \{y > 0\}$, satisfying
\begin{itemize}
\item $v \leq 1$
\item $|\{v \leq 0 \}| = a > 0$ (where the absolute value represents Lebesgue measure)
\end{itemize}
has the property that
\[ \sup_{B_{1/2} \cap \{y>0\}} v \leq \mu(a) < 1 \]
With these two facts in hand, we proceed with the proof. 

\begin{proof}
The proof is by contradiction. Let $u(0,\frac{1}{2}) = 1$. Suppose there is no $M$ which can bound values of $u$ inside the half-ball. Then $u$ achieves its maximum in it $M_0 > M$, at some point $X_0=(x_0,y_0)$. The Harnack inequality tells us that the distance to the boundary, $y_0$, satisfies
\[ y_0 \leq d_0 = \left(\frac{c}{M}\right)^{\frac{1}{p}} \]
We now proceed with a construction we repeat for each successive value of $n$, starting with $n=0$:

Consider now $B_{K d_n}(x_n,0) \cap \{y>0\}$ (the hemisphere centered on the projection of $X_n$ to the plane $y=0$), for $K$ large, greater than, say, 4. For points satisfying $y > 2 d_n$, we have, by the interior Harnack inequality, that
\[ u(X) \leq c (2d_n)^{-p} = \frac{M_n}{2^p} \]
The set $\{y > 2 d_n\}$ has measure at least a fixed fraction of $B_{K d_n}(x_n,0) \cap \{y>0\}$, independent of $K$. Thus, if we let $M_{n+1} = \sup_{B_{K d_n}(x_n,0) \cap \{y>0\}} u$, we see that inside, say, $B_{2d_n}(x_n,0) \cap \{y>0\}$, by the oscillation lemma, we have that 
\[ M_n(1-2^{-p}) \leq \sup_{B_{2 d_n}(x_n,0) \cap \{y>0\}} (u - \frac{M_n}{2^p}) \leq \mu(K) (M_{n+1} - \frac{M_n}{2^p}) \]
with $\mu \rightarrow 0$ as $K$ becomes large. Thus
\[ M_{n+1} > M_n ( 2^{-p} + \frac{1-2^{-p}}{\mu(K)}) \]
We pick $K$ sufficiently large that the factor on the right hand side is some fixed positive $\lambda > 1$. We let $X_{n+1}$ be the point where $u(X_{n+1}) = M_{n+1}$ inside $B_{K d_n}(x_n,0) \cap \{y>0\}$. 

Thus we have a sequence of points $X_n$. Notice that $K$ does not change, and hence neither does $\lambda$. As $n\rightarrow \infty$, we have
\[ u(X_n) \geq \lambda^n M_0 \rightarrow \infty\]
while 
\[ y_n \leq d_n = (cM_0)^{-\frac{1}{p}} \lambda^{-\frac{n}{p}} \rightarrow 0.\] 
The distances between the points satisfy
\[ |X_{n+1}-X_n| \leq K d_n \]
and so the sequence has
\[ d(X_0, X_n) \leq K \sum d_n  \leq K \frac{(c M)^{-\frac{1}{p}}}{1-\lambda^{-\frac{1}{p}}} \]
which can be made to converge inside $B_\frac{9}{16}$ if we take our initial $M$ sufficiently large, giving us a sequence of points $X_n$, with limit points where $u$ blows up along $y=0$. This contradicts our original assumption that $u$ continuously assumes values along the boundary $\{y=0\}$. 
\end{proof}

\subsection{The center is positive when the boundary is large}
Lemma \ref{reg-lem} consists of demonstrating that when $c_0$ is sufficiently large, a comparison subsolution which is purely positive in $B_\frac{1}{3}$ can be built, which serves as a lower barrier to the solution. 

\subsubsection{Definition of a comparison subsolution}

We seek sufficient conditions for a function to be a comparison \textit{subsolution} of our variational problem. One way to do this is to show that, for a subsolution $w$, where $u>w$, we can improve the energy: if $v = \max (u,w)$, then 
\[ J(u) - J(v) = \frac{1}{2} \int y^a \left(|\nabla u|^2  - |\nabla v|^2 \right) dy dx + \int_\Gamma u^\gamma - v^\gamma dx \geq 0 \]
since $u$ is the energy minimizer. Clearly, the second term is negative; our approach lies on setting conditions so that the first term dominates the second. 

We assume $u$, $v$ sharing the same Dirichlet boundary conditions along $\partial B$, and integrate by parts:
\begin{eqnarray*}
\int \frac{1}{2} y^a(|\nabla u|^2 - |\nabla v|^2) dx dy &=& \frac{1}{2} \int y^a (\nabla u + \nabla v) \cdot (\nabla u - \nabla v) dx dy \\
&=& -\frac{1}{2} \int (u - v) \nabla \cdot (y^a \nabla (u+v)) dx dy - \frac{1}{2} \int_\Gamma (u-v) \lim_{y\rightarrow 0} y^a \partial_y (u+v) dx 
\end{eqnarray*}
We define $v = \max(u,w)$, where $w$ satisfies
\[ \lim_{y \rightarrow 0} y^a \partial_y w \geq M w^{\gamma -1} \]
along $\Gamma$, and $\nabla \cdot (y^a \nabla w) = 0$ in $B_1$. Then $\lim_{y\rightarrow 0} \partial_y v \geq M v^{\gamma -1}$ on those portions where $v > u$, whence we can write
\[ \int y^a(|\nabla u|^2 - |\nabla v|^2) dx dy \geq \int_{\Gamma \cap \{v> u > 0\}} (v-u) (\gamma u^{\gamma -1} + M v^{\gamma -1}) dx + \int_{\Gamma \cap \{ v > u = 0\}} Mv^\gamma dx \]
Recall now that
\[ J(u) - J(v) = \frac{1}{2} \int_B y^a(|\nabla u|^2 - |\nabla v|^2) dx dy  + \int_\Gamma u^\gamma - v^\gamma dx \]
Since $u$ is the energy minimizer, we need for this term to be negative. 

We consider the functions
\[ \psi(s) = s^\gamma - t^\gamma \]
and 
\[ \phi(s) = \frac{1}{2} (s-t)(\gamma t^{\gamma -1} + M s^{\gamma -1}) \]
Clearly, $\phi(t) = \psi(t) = 0$. We now examine their behavior in the range $0 \leq t < s$. When $s > t$, 
\[\psi'(s) = \gamma s^{\gamma -1} \leq \phi'(s) = \frac{1}{2} \gamma t^{\gamma -1} + \gamma M s^{\gamma -1} + \gamma (1-\gamma) M t s^{\gamma -2}\]. Thus, $\phi(s) > \psi(s)$ when $s>t$, and we can write 
\begin{eqnarray*}
J(u) - J(v) &=& \frac{1}{2} \int_B y^a(|\nabla u|^2 - |\nabla v|^2) dx dy  + \int_\Gamma u^\gamma - v^\gamma dx \\
&\geq& \frac{1}{2}\int_{\Gamma \cap \{v> u > 0\}} (v-u) (\gamma u^{\gamma -1} + M v^{\gamma -1}) dx + \frac{1}{2} \int_{\Gamma \cap \{v > u = 0\}} M v^\gamma dx + \int_\Gamma u^\gamma - v^\gamma dx  \\
&\geq& 0 
\end{eqnarray*}
with the last equality being strict if $v$ differs from $u$ on a set with positive measure. This is satisfied if we set $M = 2$. 

Hence, whenever such a $w$ exists, we can decrease the energy of $u$, a contradiction on the definition of $u$ as the energy minimizer; that is to say, we have $u \geq w$.  

\subsubsection{Construction of such a subsolution: Proof of Lemma \ref{reg-lem}}
By the results of the previous subseciton, it suffices to construct a comparison subsolution $w$, which is positive on $B_\frac{1}{3}$ and greater than a fixed constant on $B_\frac{1}{6}$. 

We want our subsolution $w$ to have three properties: we would like our $w$ to take the same values as $u$ along $\partial B \cap \{y>0\}$, we would like it to satisfy the conditions
\[ \nabla \cdot (y^a \nabla w) = 0 \] 
in $B_+$, and along $\Gamma$ we would like
\[ \lim_{y\rightarrow 0} y^a \partial_y w \geq 2 w^{\gamma -1} \]
wherever $w>0$, and finally we want $w>0$ in $B_{\frac{1}{3}} \cap \Gamma$. We will define our $w$ in two parts. 
\[ w = w_1 + w_2 \]

We set $w_1$ by setting, for $x \in \mathbb{R}^n$, 
\[ \psi(x) = \begin{cases}
		0 & |x| > \frac{1}{3} \\
		-(1-3|x|)^{\beta-2\sigma} & |x| \leq \frac{1}{3}
             \end{cases}
\]
Let
\[ (I_{2\sigma}\psi) (x) = C_{n,\sigma} \int \frac{\psi(z)}{|x-z|^{n-2\sigma}} dz\] 
be the Riesz potential of $\psi$. We state a technical lemma relying on classical results in the theory of fractional integration and Riesz potentials, which leave to appendix \ref{ap:Riesz}:
\begin{lem}\label{lem:Riesz}
$(I_{2\sigma}\psi)(x)$ is well defined and continuous as a function, radial, has fractional Laplacian equal to $\psi(x)$, and furthermore, there exists $\delta > 0$ such that
\[ |(I_{2\sigma}\psi)(r) - (I_{2\sigma}\psi)(\frac{1}{3})| \leq C (1-3r)^\alpha \]
for $\frac{1}{3} > r = |x| > \frac{1}{3}-\delta$ where $\min(\beta,1) > \alpha > \sigma$.  
\end{lem}

We let $b(x)$ be equal to $I_{2\sigma} \psi$ on $\mathbb{R}^n \setminus B_{\frac{1}{3}}$, and have $(-\Delta)^{\sigma} b = 0$ inside $B_\frac{1}{3}$, and then we set 
\[ \tilde{w}(x) = (I_{2\sigma} \psi)(x) - b(x) \]
$b$ is the solution to the standard Dirichlet problem for the fractional Laplacian; its existence is guaranteed by the standard theory (see, e.g., Landkof \cite{Landkof}). Notice that $\tilde{w}$ is $\sigma$-subharmonic. This means it is negative inside $B_{\frac{1}{3}}$, and 0 outside of it. Furthermore, the maximum principle for $\sigma$-harmonic functions (Lemma \ref{lem:Hopf}) applied to $b(x)$ tells us that there is a constant such that
\[ |b(x) - b(\frac{1}{3})| \leq C (1-3r)^\sigma \]

Now we let
\[ w_1(x,y) = C_{n,\sigma} \int \frac{y^{2\sigma} \tilde{w}(z)}{\left((x-z)^2 + y^2\right)^\frac{n+2\sigma}{2}} dz \]
where $z$ ranges over $\mathbb{R}^n$. This is, of course, the Poisson kernel for the fractional Laplacian convolved with $\tilde{w}$, giving us a $w_1$ that satisfies $\nabla \cdot (y^a \nabla w_1) = 0 $ in the interior, which takes on the values of $\tilde{w}$ along $\{y=0\}$, satisfying $\lim_{y\rightarrow 0} y^a \partial_y w_1(x,0) = \psi(x)$ (by the extension result of Caffarelli and Silvestre \cite{C-S}). 

For the sake of future estimations, it is helpful to bound $-w_1$ from above by an auxiliary function. We let $q_0 = 2 \sup (-w_1)$, and we let 
\[ q(x) = \begin{cases}
           q_0 & |x| < \frac{1}{3}-\delta \\
           0 & |x| > \frac{1}{3}
          \end{cases}
\]
and let $q$ satisfy $(-\Delta)^\sigma q = 0$ on the annular ring $\frac{1}{3}-\delta < |x| < \frac{1}{3}$. The comparison principle for fractional-harmonic functions then tells us that $q \geq -w_1$ on $\Gamma$. We extend $q$ to $\mathbb{R}^{n+1}_+$ in the usual way via the Poisson kernel. 
\[ Q(x,y) =  C_{n,\sigma} \int \frac{y^{2\sigma} q(z)}{\left((x-z)^2 + y^2\right)^\frac{n+2\sigma}{2}} dz \]
Thus, we have
\begin{prop}
 $Q(x,y) \geq |w_1(x,y)|$ in the upper half-space $\mathbb{R}^{n+1}_+$. 
\end{prop}

We set $w_2$ with boundary conditions
\begin{equation*} 
 w_2(X) = \begin{cases} 
	u - w_1(X) & X \in {\partial B_1 \cap \{y>0\}} \\
        0 & X \in {\Gamma \setminus B_{\frac{1}{3}}}
       \end{cases}
\end{equation*}
and let it satisfy the problem
\[ \nabla \cdot ( y^a \nabla w_2) = 0 \]
when $X \in B_+$, and 
\[ \lim_{y\rightarrow 0} y^a \partial_y w_2 = 0 \] 
when $X \in \Gamma \cap B_{\frac{1}{3}}$. 

Now we need to estimate properties of $w= w_1+w_2$, which we do by comparing $w_2$ to $Q$
\begin{prop}
For any $\lambda > 0$, a sufficiently large value of $c_0$ will ensure that $w_2(X) \geq \lambda Q(X)$ inside $B_+$
\end{prop}
\begin{proof}
Since both functions satisfy $\nabla \cdot y^a (\nabla v) = 0$ inside $B_+$, it suffices to examine their relative behavior along $\partial_{B_1} \cap \{y > 0\}$ and along $\Gamma$. 

Along $\partial_{B_1}\cap \{y > 0\}$, the boundary comparison principle (see \cite{F-K-J-2} for a proof in the case of $A_2$ weighted degenerate elliptic equations) tells us that 
\[ u(x,y) \geq C c_0 y^{2\sigma} \]
since $u \geq c_0$ when $y> \frac{1}{2}$. We also have from the formula that that $Q(x,y) \leq C q_0 y^{2\sigma}$. Hence, we just need $c_0$ to be sufficiently large.

The behavior along $\Gamma$ is a touch trickier. We divide our analysis of the behavior of $w_2$ along $\Gamma$ into two parts: the first part concerns the interior of $B_{\frac{1}{3} -\delta}$, where $\delta$ is from Lemma \ref{lem:Riesz}, and the other in the thin annular ring $\frac{1}{3}-\delta < |x| < \frac{1}{3}$. Clearly, $\Gamma \setminus B_{\frac{1}{3}}$ is not taken care of, since $w_2$ and $Q$ are identically 0 there. 

We note that $w_2|_{\partial B \cap \{y>0\}} > 0$, so that we can apply the Harnack inequality in the interior. 
We bound $w_2$ from below by a function $\hat{w_2}$, which we define as follows: let 
\begin{equation*} 
 \hat{w_2}(X) = \begin{cases} 
	u - w_1(X) & X \in {\partial B_1 \cap \{y>0\}} \\
        0 & X \in {\Gamma}
       \end{cases}
\end{equation*}
and let it satisfy the problem
\[ \nabla \cdot ( y^a \nabla \hat{w_2}) = 0 \]
Clearly $0\leq \hat{w_2} \leq w_2$ in the domain. Since we know $w_2$ in $\{y > \frac{1}{2} \}$ is greater than $c_0$, it follows that so too  $\hat{w_2}$ at interior points, such as, say, $X=(0,\frac{1}{6})$, is linear in $c_0$, and hence so is $w_2$. We apply the Harnack inequality to $w_2$ inside the ball $B_\frac{1}{3}$ to see that $w_2 \geq C c_0$ in $B_{\frac{1}{3}-\delta}$ can be made as large as we wish, where $\delta$ is from lemma \ref{lem:Riesz}. Thus, inside $B_{\frac{1}{3}-\delta} \cap \Gamma$, we can choose $c_0$ so that $w_2 \geq Q$. 

In the annular ring proper, both $Q(x,y)$ and $w_2(x,y)$ satisfy $\lim_{y\rightarrow 0} y^a \partial_y v = 0$, whence we can invoke the Hopf lemma to see that $w_2 \geq Q$. 
\end{proof}

\begin{cor}
\[ w = w_1 + w_2 \geq (\lambda -1) Q \]
\end{cor}
and hence by making $\lambda$ sufficiently large we can make $w \geq C q_0$ in $B_{\frac{1}{6}}$. 

We close our construction with a lemma, which shows that $w$ has all the desired properties of a subsolution. 

\begin{lem}\label{lem:positive}
\[ w_1 + w_2 \geq 0 \] 
and for $\lambda$ from the previous proposition sufficiently large (which is really to say for $c_0$ sufficiently large), we have 
\[ \lim_{y\rightarrow 0} y^a \partial_y w \geq 2 w^{\gamma -1} \]
along $\Gamma$, wherever $w \ne 0$. 
\end{lem}
\begin{proof}
Since $w_1 \geq -Q$ and $w_2 \geq \lambda Q$, we have 
\[ w = w_1 + w_2 \geq w_2 - Q \geq 0 \]

On $\Gamma \cap B_{\frac{1}{3} - \delta}$, we have 
\begin{eqnarray*}
\lim_{y\rightarrow 0} y^a \partial_y w &=& \lim_{y\rightarrow 0} y^a \partial_y w_1 \\
&=& (1-3|x|)^{\beta - 2\sigma}  \\
&\geq& (3\delta)^{\beta - 2\sigma} \\
\end{eqnarray*}
By setting $\lambda$ sufficiently large, we can attain
\begin{eqnarray*}
(3\delta)^{\beta - 2\sigma} &\geq& 2 (\lambda q_0)^{\gamma -1} \\
&\geq & 2 w_2^{\gamma -1} \\
&\geq & 2 w^{\gamma-1}
\end{eqnarray*}

On the annular ring, we invoke Lemma \ref{lem:Hopf} to see that there is a constant $c$ such that 
\[ q(x) \geq c (1-3|x|)^{\sigma} \]
whence we derive the relation
\[ w(x) \geq (\lambda -1) c (1-3|x|)^\sigma \]
By setting $\lambda$ sufficiently large, we can make
\[ 2^{\frac{1}{\gamma -1}} (\lambda-1) c (1-3|x|)^\sigma \geq (1-3|x|)^\beta \]
in the entire annular ring, whence we can attain 
\begin{eqnarray*}
\lim_{y\rightarrow 0} y^a \partial_y w &=& \lim_{y\rightarrow 0} y^a \partial_y w_1 \\
&=& (1-3|x|)^{\beta - 2\sigma}  \\
&=& (1-3|x|)^{\beta(\gamma -1)}  \\
&\geq& 2 ((\lambda-1) c (1-3|x|)^\sigma)^{\gamma -1} \\
&\geq& 2 w^{\gamma -1} 
\end{eqnarray*}
\end{proof}

\subsection{$C^\beta$ estimates for $u$}\label{subsec:c^beta-estimates}
The goal of this subsection is to provide a proof for Corollaries \ref{thm:c-beta} and \ref{thm:global-beta}. 

For the first, we will do this by analyzing the effective equation satisfied by $u$, restricted to $\Gamma$, in the neighborhood of a free boundary point. The estimates follow the spirit of the analysis conducted in Section III of \cite{P-1}: we will first show that appropriate Holder norms of $u$ satisfy certain pointwise estimates in terms of the value of $u$ itself, and then put these estimates together to obtain a uniform $C^\beta$ estimate. 

For the second, we will follow a similar procedure, first using interior estimates to get pointwise bounds on $\nabla u$ when $y > 0$, and then tie these together with the $C^\beta$ estimate along $\{y=0\}$ to get a uniform $C^\beta$ estimate. 

\subsubsection{Along $\{y=0\}$}
\begin{lem}\label{lem:c-beta}
$\{u>0\} \cap \Gamma$ is open with respect to $\mathbb{R}^n$, on which $u$ satisfies 
\[ \lim_{y\rightarrow 0} y^a \partial_y u = \gamma u^{\gamma -1} \]
and furthermore $u\in C^\infty(\{u>0\})$, such that the tangential derivatives of $u$, which we represent by $\nabla_x u$, satisfy
\[ |\nabla_x u(p,0)| \leq C (u(p,0))^{\frac{\beta -1}{\beta}} \]
and, moreover, that the tangential second derivatives of $u$, which we represent by $\nabla_{xx} u$, satisfy
\[ |\nabla_xx u(p,0)| \leq C (u(p,0))^{\frac{\beta -2}{\beta}} \]
where $p$ is any point in $ \{u> 0\} \cap \{y=0\}$. 
\end{lem}

\begin{proof}
If $p \in \Gamma$ is some point such that $u(p,0) > 0$, then the variant Boundary Harnack inequality tells us that $u(p,\delta) \geq C u(p,0)$. If we make the usual dilation by $\lambda$ about $p$, we see that $u_\lambda(x,y) = \frac{1}{\lambda^\beta} u(\lambda(x-p)+p,\lambda y)$ satisfies 
\[ u_\lambda(p,1) > c \frac{1}{\lambda^\beta} u(p,0) \]
For $\lambda$ sufficiently small, $c \lambda^{-\beta} u(p,0)$ can be made larger than the constant needed in Lemma \ref{reg-lem}, whence $u_\lambda \geq C > 0 $ in $B_{\frac{1}{6}}(p,0)$, or, in the original $u$, we can say $u \geq C \lambda^\beta > 0$ in $B_\frac{\lambda}{6} (p,0)$ for $\lambda$ sufficiently small. 

Hence, the set $u>0$ is open with respect to $\Gamma$, and, on every set $D$ compactly contained within $\{u>0\}$, is bounded away from zero, hence 
\[ \lim_{y\rightarrow 0} y^a \partial_y u = \gamma u^{\gamma -1} \in L^\infty(D) \] 

Consider now $B=B_\frac{1}{6}(p,0)$. We let $u_1$ be the Riesz potential of $-\gamma u_\lambda^{\gamma -1} \chi_{B}$, and set $u_2 = u_\lambda - u_1$. Then on $B$ we have
\[ \lim_{y\rightarrow 0} y^a \partial_y u_1 = \gamma u_\lambda^{\gamma -1} \]
while 
\[ \lim_{y\rightarrow 0} y^a \partial_y u_2 = 0 \]
and both of the $u_i$ satisfy $\nabla \cdot (y^a \nabla u_i) = 0$. 

Since the tangential derivatives of $u_2$ also satisfy the same equations as $u_2$, we have that $u_2 \in C^\infty(B)$, with 
\[ |\nabla_x u_2 (p,0)| \leq C \]
and 
\[ |\nabla_{xx} u_2 (p,0)| \leq C \]
in $B_{\frac{1}{8}}(p,0)$, by the estimates found in \cite{C-S-S}. Similarly, we can use the potential-theoretic estimates found in \cite{S-1} iteratively to show that $u_1 \in C^\infty(B)$, with 
\[ |\nabla_x u_1 (p,0)| \leq C \]
and 
\[ |\nabla_{xx} u_1 (p,0)| \leq C. \]

Hence, after rescaling, we can say that, for the tangential derivatives of $u$, we have 
\[ |\nabla_x u(p,0)| \leq C \lambda^{\beta -1} \]
and 
\[ |\nabla_{xx} u(p,0)| \leq C \lambda^{\beta -2} \]

How small need $\lambda$ be? Our condition was that $c \lambda^{-\beta} u(p,0) \geq c_0$, whence we see that $\lambda = (C u(p,0))^\frac{1}{\beta}$ suffices. The conclusion follows. 
\end{proof}

Up to now, it has been possible to treat the cases where $\beta \geq 1$ and $\beta < 1$ as if they were the same. For the remaining two theorems, we have to recognize the difference. The result here is proved very much the style of \cite{P-1} and \cite{C-K}. 
\begin{thm}\label{thm:c-beta-k}
Suppose $\beta < 1$. Then there exists a $K = K(\delta, n, \beta)$, such that if $x_1,x_2 \in \mathbb{R}^n$ are in a $\delta$-neighborhood of the free boundary, we have 
\[|u(x_1,0)-u(x_2,0)| \leq K |x_1-x_2|^\beta \]
If $\beta \geq 1$, there exists a $K = K(\delta, n, \beta)$, such that if $x_1,x_2 \in \mathbb{R}^n$ are in a $\delta$-neighborhood of the free boundary, we have 
\[ |\nabla_x u(x_1,0)- \nabla_x u(x_2,0)| \leq K |x_1-x_2|^{\beta-1} \]
\end{thm}
In either case, since away from the $\delta$-neighborhood of the free boundary, $u\in C^\infty$, this means we can put the two together to get a uniform $C^\beta$ norm for $u$. 
\begin{proof}
As in the previous lemma, we notice that there is a constant $C_1$, such that if $u(x_1,0) \geq C_1$, then the variant boundary Harnack inequality tells us that $u$ satisfies the conditions for Lemma \ref{reg-lem}, and hence $u \geq C_2$ inside $B_\frac{1}{6}(x_1,0)$. Rescaling this statement, we have that if $u(x_1,0) \geq C_1 r^\beta$, then $u \geq C_2 r^\beta$ inside $B_\frac{r}{6}(x_1,0)$. 

We now consider three cases:
\begin{enumerate}
 \item $u(x_1,0) \geq C_1 (6 |x_1-x_2|)^\beta$ and $|x_1-x_2| < \frac{\delta}{4}$
 \item $u(x_1,0) \geq C_1 (6 |x_1-x_2|)^\beta$ and $|x_1-x_2| \geq \frac{\delta}{4}$
 \item $u(x_1,0), u(x_2,0) \leq C_1 (6 |x_1-x_2|)^\beta$
\end{enumerate}

\textit{1)} Consider the line segment joining $x_1$ and $x_2$, with $r = 6 |x_1-x_2|$. Since $u(x_1,0) \geq C_1 r^\beta$, we have $u(x,0) \geq C_2 r^\beta$ inside $B_\frac{r}{6}(x_1,0)$, which happily is precisely $B_{|x_1-x_2|}(x_1,0)$. Hence, when $\beta < 1$, the mean value theorem applied along this line segment tells us
\[ |u(x_1,0) - u(x_2,0)| \leq |\nabla_x u(x',0)| |x_1-x_2| \]
where $x'$ is some point along our line segment. By applying the estimates from Lemma \ref{lem:c-beta}, we have 
\[ |\nabla_x u(x',0)| |x_1-x_2| \leq C (u(x',0))^{\frac{\beta -1}{\beta}} |x_1-x_2| \leq C |x_1 - x_2|^\beta \]
When $\beta \geq 1$, we consider instead 
\[ |\nabla_x u(x_1,0) - \nabla_x u(x_2,0)| \leq |\nabla_{xx} u(x',0)| |x_1-x_2| \]
and by applying the estimates on the tangential second derivatives, we have
\[ |\nabla_{xx} u(x',0)||x_1-x_2| \leq C (u(x',0))^{\frac{\beta -2}{\beta}} |x_1-x_2|\leq C |x_1 - x_2|^{\beta-1} \]

\textit{2)} In this case, we simply say directly that, if $\beta < 1$, 
\begin{eqnarray*}
|u(x_1,0) - u(x_2,0)| &\leq& |u(x_1,0)| + |u(x_2,0)| \\
&\leq& C \delta^\beta \leq C |x_1-x_2|^\beta 
\end{eqnarray*}
where we invoke Theorem \ref{thm:opt-reg} on the last step. 
If $\beta \geq 1$, we say 
\begin{eqnarray*}
|\nabla_x u(x_1,0) - \nabla_x u(x_2,0)| &\leq& |\nabla_x u(x_1,0)| + |\nabla_x u(x_2,0)| \\
&\leq& C  (u(x_1,0))^{\frac{\beta -1}{\beta}} + C(u(x_2,0))^{\frac{\beta -1}{\beta}} \\
&\leq& C \delta^{\beta -1}
\end{eqnarray*}
where we invoke Theorem \ref{thm:opt-reg} on the last step. 

\textit{3)} The calculations are exactly like case 2, only instead of invoking Theorem \ref{thm:opt-reg} to bound $u$ pointwise, we invoke the hypothesis. 
\end{proof}

From this result, Corollary \ref{thm:c-beta} is obvious. 
\subsubsection{The estimates when $y>0$}

Note first that inside $y > 0$, $u \in C^\infty$, since $\nabla \cdot (y^a \nabla u) = 0$ is uniformly elliptic with smooth coefficients on any compact subset contained within $\{y > 0\}$ (with differing ellipticities, of course). We can thus assume that $u$ is smooth far away, and concentrate on its behavior for small values of $y$. 

We start with an elementary lemma that gives us pointwise estimates on the derivatives of $u$ via rescaling:
\begin{lem}\label{lem:point-gradient-g}
Let $u$ be a non-negative function satisfying $\nabla \cdot (y^a \nabla u) = 0$ inside $B_R \cap \{y>0\}$ for some large $R$ . Then there is a constant $C$ depending only on $n$ and $\sigma$, such that
\[ |\nabla u (x_0,y_0)| \leq \frac{C}{y_0} u(x_0,y_0) \]
and 
\[ |D^2 u (x_0,y_0)| \leq \frac{C}{y_0^2} u(x_0,y_0) \]
whenever $y_0 > 0$ and $B_{\frac{y_0}{2}}((x_0,y_0)) \subset B_R$. 
\end{lem}
\begin{proof}
Suppose first that $y_0 = 1$ and $B_{\frac{1}{2}}(x_0,1)$ is inside $B_R$. Then inside this ball, $y^a$ is a bounded, $C^\infty$ coefficient, so the standard regularity theory for weak solutions gives us the estimates
\[ |\nabla u (x_0,1)| \leq C u(x_0,1) \]
and 
\[ |D^2 u(x_0,1)| \leq C u(x_0,1). \]
For general $y$, we simply consider the rescaling $w(x,y) = u( x_0 + (x-x_0) y_0, y_0 y)$ and write the estimate for $w$ in terms of $u$. 
\end{proof}

Next, we provide a boundary estimate on the growth of $u$ away from the line $y=0$. We choose nice constants for the varius radii and the lines, bearing in mind that we can rescale.  
\begin{lem}\label{lem:beta-growth-g}
Let $u$ be an energy minimizer inside $B_8$ with nontrivial free boundary. Then there exists a constant $C$ such that, for $(x,y) \in B_3$, we have  
\[ |u(x,y) - u(x,0)| \leq C y^\beta \]
\end{lem}
\begin{proof}
If $u(x,0) = 0$ for any $(x,0) \in B_3$, then Theorem \ref{thm:opt-reg} suffices for that value of $x$. Thus we only need consider values of $x$ such that $u(x,0) > 0$. For these values of $x$, we split $u$ into two parts. 

Let $\phi(x)$ be a mollifier compactly supported on $B_6$ which is 1 on $B_4$. Then split $u$ into two parts. The first is the extension via the Poisson kernel of the values of $u$ along $\{y=0\}$, 
\[u_1(x,y) = P_y(x) * (u(x,0)\phi(x)) \]
with
\[ u_2(x,y) = u(x,y)-u_1(x,y). \]
$u_2$ satisfies that $u_2(x,0) = 0$ inside $B_4$, and $\nabla \cdot (y^a \nabla u_2) = 0$ whenever $y>0$. That there is a nontrivial free boundary and Theorem \ref{thm:opt-reg} provides an upper bound for $|u|$ inside $B_6$, and hence for $u_1$ as well. Thus we can apply the maximum principle to $u_2$, and conclude that there exists a constant $C$ such that
\[ |u_2(x,y)| \leq C y^{2\sigma}. \]

The argument for $u_1$ relies on properties of the Poisson kernel (this argument follows \cite{Stein}, prop 4.7). A bit of calculation tells us that 
\[ \int_{\mathbb{R}^n} \left|\frac{\partial P_y}{\partial y} (x,y) \right|  dx \leq \frac{C}{y} \]
whence we can write 
\[ \frac{\partial u_1}{\partial y} (x,y) = \int_{\mathbb{R}^n} \frac{\partial P_y}{\partial y}(z) \left(u(x-z,0) - u(x,0)\right) dz \]
Applying the $C^\beta$ estimate for $u(x,0)$, we find
\[ \left|\frac{\partial u_1}{\partial y} (x,y) \right| \leq C \|u(x,0)\|_{C^\beta} y^{\beta - 1}. \]

Hence we can conclude that 
\[ |u(x,y) - u(x,0)| \leq |u_1(x,y)-u_1(x,0)| + |u_2(x,y) - u_2(x,0)| \leq C y^\beta + C y^{2\sigma} \]
Since $\beta < 2\sigma$, we have the desired result. 

\end{proof}


With these two lemmata in hand, we can prove the analogue of Theorem \ref{thm:c-beta-k} for the domain where $y>0$. 
\begin{thm}\label{thm:global-beta-k}
Suppose $\beta < 1$. Then there exists a $K = K(\delta, n, \beta)$, such that if $X_1=(x_1,y_1),X_2=(x_2,y_2) \in \mathbb{R}^{n+1}_+$ are in a $\delta$-neighborhood of the free boundary, we have 
\[|u(X_1)-u(X_2)| \leq K |X_1-X_2|^\beta \]
If $\beta \geq 1$, there exists a $K = K(\delta, n, \beta,\alpha)$, such that if $X_1,X_2 \in \mathbb{R}^{n+1}_+$ are in a $\delta$-neighborhood of the free boundary, we have 
\[|u(X_1)-u(X_2)| \leq K |X_1-X_2|^{\alpha} \]
for any $\alpha < 1$. 
\end{thm}
\begin{proof}
Without loss of generality, assume that $y_1 \leq y_2$. 

First, assume that $\beta < 1$. Suppose that $y_2\leq |X_1-X_2|^{1-\beta}$ . Then, using the $C^\beta$-regularity of $u$ restricted to $y=0$ and the previous lemma, we write that 
\begin{eqnarray*}
|u(x_1,y_1) - u(x_2,y_2)| &\leq& |u(x_1,y_1)-u(x_1,0)| + |u(x_2,y_2) - u(x_2,0)| + |u(x_1,0)-u(x_2,0)| \\
&\leq& C (y_1^\beta + y_2^\beta) + C |x_1-x_2|^\beta \\
&\leq& 2 C |X_1-X_2|^{\frac{\beta}{1-\beta}} + C |X_1 - X_2|^\beta \\
&\leq& C |X_1-X_2|^\beta
\end{eqnarray*}

On the other hand, if  $y_1 \geq |X_1-X_2|^{1-\beta}$, then we use our pointwise gradient estimates and the special properties of this case to write that
\begin{eqnarray*}
 |u(X_1) - u(X_2)| &\leq& |\nabla u(\tilde{X})| |X_1-X_2| \\
&\leq& \frac{u(\tilde{X})}{y_1} |X_1-X_2| \\
&\leq& \frac{C}{|X_1-X_2|^{1-\beta}} |X_1-X_2| \leq C |X_1-X_2|^\beta
\end{eqnarray*}
where $\tilde{X}$ is some point on the line joining $X_1$ and $X_2$. 

If $y_1 \leq |X_1-X_2|^{1-\beta}$ and $y_2 \geq |X_1-X_2|^{1-\beta}$, then we consider: 
\[ |u(X_1)-u(X_2)| \leq |u(x_1,y_1) - u(x_2,|X_1-X_2|^{1-\beta})| + |u(x_2,|X_1-X_2|^{1-\beta}) - u(x_2,y_2)| \]
The first term is controlled by the first method above, and the second term is controlled by the second method.

For the case when $\beta \geq 1$, simply let replace $|X_1-X_2|^{1-\beta}$ in the preceding argument by $|X_1-X_2|^{1-\alpha}$. 
\end{proof}
\section{Non-degeneracy}
Our goal in this section is to prove that energy minimizers of
\[ J(u) = \frac{1}{2} \int_{B_+} y^a |\nabla u|^2 dx dy + \int_\Gamma u^\gamma dx \]
possess the property they are non-degenerate, which is to say that near the free boundary, they grow away from 0, and do not stay small. To be precise, our final theorem is
\begin{thm}\label{thm:non-degeneracy}
Let 0 be a point of the free boundary of $u$, a minimizer of $J(u)$. Then there exists a constant $C > 0$ such  
\[ \sup_{B_r} u \geq C r^\beta \]
\end{thm}
Our strategy for proving this theorem is first to show that at a fixed distance away from the free boundary, there is a point which attains the desired growth. 

\begin{thm}\label{thm:growth}
Let $x_0 \in \Gamma$ be a point such that $d(x_0,F(u)) = r$, where $F(u)$ is the free boundary. Then, there exists a universal constant $\tau(n, \sigma, \gamma) >0$ such that 
\[ u(x_1) \geq \tau r^\beta \]
where $|x_0-x_1|\leq \frac{r}{4}$, and $x_1 \in \Gamma$. 
\end{thm}
\begin{proof}
As is typical, we shall rely on the scaling property of energy minimizers, specifically, that on $\lambda B_+$, $\frac{1}{\lambda^\beta} u(x_0 + \lambda X)$ is still an energy minimizer. Hence, we can assume $d(x_0, F(u)) = 1$ and we only need to show that there exists an $x_1$ with 
\[ u(x_1) \geq c_1 \]
for some $x_1$ with $|x_0-x_1|\leq \frac{1}{4}$, $x_1 \in \Gamma$. 

The standard Green's identity applied to some test function $\phi \geq 0$ with support compactly contained within $B_1(x_0)$ tells us that 
\[ \int_\Gamma u (\lim_{y\rightarrow 0} y^a \partial_y \phi) - \phi (\lim_{y\rightarrow 0} y^a \partial_y u) dx = - \int_{B_\frac{1}{2}(u_0) \cap \{y>0\}} u \nabla \cdot (y^a \nabla \phi) dx dy \]
We notice that, along $\Gamma \cap B_{\frac{1}{2}}(x_0)$, we have $u>0$, and hence 
\[ \lim_{y\rightarrow 0} y^a \partial_y u = \gamma u^{\gamma -1} \]

Thus, we attain the condition that 
\[ \left| \int_\Gamma \gamma u^{\gamma -1} \phi dx \right| \leq \left| \int_\Gamma u \lim_{y\rightarrow 0} y^a \partial_y \phi dx \right| + \left| \int u(x) \nabla \cdot (y^a \nabla \phi) dx dy \right|  \]

Now let us suppose to the contrary that there is no constant $c_1$, that is to say, for any $\epsilon > 0$ there is a minimizer such that $|u| \leq \epsilon$ inside $B_\frac{1}{2}(x_0) \cap \Gamma$. Since $d(x_0,F(u)) = 1$, we have $d(B_{\frac{1}{2}}(x_0), F(u)) \leq \frac{3}{2}$, and we apply optimal regularity to bound the interior term: on $B_\frac{1}{2}(x_0) \cap \{y>0\}$, we have
\[ u(x) \leq C \]

Putting these conditions together, and we get the argument that
\[ \left| \int_\Gamma \gamma \epsilon^{\gamma -1} \phi dx \right| \leq \left| \int_\Gamma \epsilon \lim_{y\rightarrow 0} y^a \partial_y \phi dx \right| + \left| \int C \nabla \cdot (y^a \nabla \phi) dx dy \right|  \]
for arbitrarily small $\epsilon$. Since the left hand side becomes very large and the right hand side is bounded, we have a contradiction: $u$ cannot be made uniformly arbitrarily small inside $B_\frac{1}{2}(x_0) \cap \Gamma$, and thus there exists a constant $\tau$ such that $u > \tau$ at some point on $B_\frac{1}{2}(x_0) \cap \Gamma$, which we call $x_1$. 
\end{proof}

Now we begin the proof of Theorem \ref{thm:non-degeneracy}, which is essentially identical to that given in \cite{C-R-S}, and reproduced here for completeness:
\begin{proof}[Proof of Theorem \ref{thm:non-degeneracy}]
The proof is divided into two steps. 

\noindent
\textit{Step 1}. Let $u$ be a local minimizer in $B_M$ such that
\begin{itemize}
 \item 0 is a free boundary point,
 \item $B_1(e_1,0) \cap \Gamma \subset \{u> 0\} \cap \Gamma$,
 \item $u(e_1,0) = \tau > 0$ where $\tau$ is the constant from Theorem \ref{thm:growth}, known to be bounded both from above and from below away from 0. 
\end{itemize}
We claim the existence of $\lambda > 0$ and $M>0$ universal, the latter being large, such that 
\[ \sup_{B_M \cap \Gamma} u \geq (1+ \lambda) \tau \]
Suppose not. This implies the existence of a sequence of energy minimizers for our problem, $(u_k)_{k\in \mathbb{N}}$, satisfying the three listed conditions, such that 
\[ \lim_{k\rightarrow \infty} \sup_{B_M \cap \Gamma} u = \tau \]
From our regularity theorems, the family $(u_k)_k$ is equicontinuous, and may be assumed to converge uniformly on every compact subset of $\mathbb{R}^{n+1}_+$ to a function $u_\infty$ which satisfies $\lim_{y\rightarrow 0} y^a \partial_y u_\infty \geq 0$. Moreover, $u_\infty(\cdot,0)$ has a maximum at $e_1$, thus it is constant from the maximum principle. Hence $u_\infty \equiv \tau$, a contradiction because 0 is a free boundary point. 

\noindent
\textit{Step 2}. Assume that 0 is a free boundary point. As in \cite{Caffarelli-Salsa}, we construct inductively a sequence of points $(x_m)_m \in \mathbb{R}^n$, such that 
\begin{itemize}
 \item $u(x_{m+1},0) \geq (1+\lambda) u(x_m,0)$
 \item If $r_m = d(x_m, \{u=0\})$ and $\tilde{x_m}$ is a free boundary point realizing the distance, we have $x_{m+1} \in B_{Mr_m}(\tilde{x_m})$ with $u(x_{m+1},0) \geq \tau r_m^\beta$. This is from the construction of Step 1 applied to the rescaling $\frac{1}{r^\beta_m} u(\tilde{x_m} + r_m x, r_m y)$.
\end{itemize}
In particular, we have 
\[ |x_{m+1} - x_m| \leq 2(M+1)r_m \]
We end the induction at the first point $x_m$ which leaves $B_1$. This is possible, since the sequence $u(x_m,0)$ grows geometrically in $m$, but is controlled by optimal regularity considerations. Let $m_0$ be the index of the first point to leave $B_1$. Then we write
\begin{eqnarray*}
u(x_{m_0+1},0) &=& \sum_{m=0}^{m_0} (u(x_{m+1},0) - u(x_m,0)) \geq \lambda \sum u(x_m,0) \\
&\geq& C \lambda \sum d(x_m, \{u=0\} \cap B_1)^\beta \text{\,\,by Theorem \ref{thm:growth}} \\
&\geq& C' \sum |x_{m+1}-x_m|^\beta \\
&\geq& C'' \lambda \sum |x_{m+1}-x_m| \text{\,\, because $|x_{m+1}-x_m| \leq 1$} \\
&\geq& C''' 
\end{eqnarray*}
The last step is justified because $C'', \lambda$ are both universal, and $m_0$ is bounded universally by the geometric growth of the construction. Hence, for all $r > 0$, we have 
\[ \sup_{B_{Mr}} u \geq C''' r^\beta \]
which by rescaling $Mr$ to $r$ was precisely what we set out to prove. 
\end{proof}

\begin{cor}
In terms of $n$-dimensional Lebesgue measure, the positivity set $\{u >0\}$ has positive density, bounded away from 0, in a neighborhood of any free boundary point. That is to say, 
\[ \frac{|B_r \cap \{u>0\}|}{|B_r|} \geq \delta(n,\sigma,\gamma) > 0 \]
for any ball $B_r$ centered about a free boundary point. 
\end{cor}
\begin{proof}
This is a consequence of nondegeneracy, which says that a sufficiently positive point exists, and of the Holder continuity of $u$ (Theorem \ref{thm:c-beta-k}). 
\end{proof}

\section{Acknowledgements}
The author would like to thank his advisor, Luis Caffarelli, for posing the problem and for his guidance. Thanks are also due to the referee, who suggested numerous improvements. The author was supported by the (NSF-funded) Research Training Group in Applied and Computational Mathematics at the University of Texas at Austin (NSF Award No. 0636586), and a NSF Mathematical Sciences Postdoctoral Fellowship (Award No. 1103786) during the preparation of this paper. 

\appendix

\section{Proof of Lemma \ref{lem:Riesz}}\label{ap:Riesz}
That the Riesz potential of a radial function is radial is obvious from symmetry considerations. 

The proof of this lemma uses facts from the theory of Riesz potentials. The key facts we will use are given by the following theorem of Adams \cite{Adams} (pp 772):
\begin{thm}\label{thm:Adams}
If $f \in L^p(\mathbb{R}^n)$, $1 \leq p < \infty$, then 
\begin{enumerate}
 \item $I_{\alpha_1} f \in$ BMO if and only if $M_{\alpha_1} f \in L^\infty(\mathbb{R}^n)$
 \item $I_{\alpha_1} f \in$ BMO implies $I_{\alpha_1 + \alpha_2} f \in C^{\alpha_2}$ where $0 < \alpha_2 < 1$
\end{enumerate}
where 
\[ (M_{\alpha_1} f)(x) = \sup_{r>0} r^{\alpha_1 - n} \int_{B_r(x)} |f(z)| dz \]
is the fractional Hardy-Littlewood maximal function. 
\end{thm}

The plan is to set $\psi(x) = (1-3|x|)^{\beta-2\sigma} \chi_{B_{\frac{1}{3}}}(x)$, show that $\psi \in L^1(\mathbb{R}^n)$, and subsequently that $M_{\beta -2\sigma} \psi \in L^\infty$, whence we can apply the theorem to get the desired result. 

First, we prove that $\psi \in L^1$:
\begin{eqnarray*}
  \int \psi(x) dx &=& C_n \int_0^\frac{1}{3} (1-3r)^{\beta -2\sigma} r^{n-1} dr \\
&\leq& C_n \int_0^{\frac{1}{3}} (1-3r)^{\beta - 2\sigma} dr \leq C 
\end{eqnarray*}
since $\beta - 2\sigma > -\sigma > -1$. 

Next, we consider the fractional maximal function. It is clear that the points of concern lie directly atop the singularity, that is, $r = \frac{1}{3}$. In a balls $B_\rho$ about such a point, we see that 
\[ \int_{B\rho(\frac{1}{3})} \psi (x) dx \leq C \rho^{n+\beta - 2\sigma} \]
which is precisely the scaling needed to see that $M_{\beta-2\sigma} \psi \leq C$. 

Hence, we can apply the theorem of Adams and we conclude our lemma. 

\section{Some facts concerning degenerate elliptic equations with $A_2$ weights}\label{ap:degiorgi}
These facts are classical for the case of uniformly elliptic divergence form equations (see, e.g., \cite{Lin}), and close analogues are apparent by following the work of Fabes, Kenig, and Serapioni \cite{F-K-S}, although they are not explicitly stated there. We demonstrate the connection, and give precise statements, in the case of the equation
\[ \nabla \cdot (A \nabla u) = 0 \]
where $w(x) \lambda |\xi|^2 \leq \xi^T A \xi \leq w(x) \Lambda |\xi|^2$ and $w$ is some $A_2$ weight.  
\subsection{The De Giorgi oscillation lemma}
\begin{lem}
Suppose $u$ is a positive supersolution in $B_{2}$ with 
\[ |\{ x \in B_1; u \geq 1 \}| \geq \epsilon |B_1| \]
Then there exists a constant $C$ depending only on $\epsilon, n,$ and $\sigma$ such that 
\[ \inf_{B_\frac{1}{2}} u \geq C \]
\end{lem}
Although we prove the former statement, the obvious corollary concerning subsolutions follows from applying the lemma to $1-u$, and is what we actually use:
\begin{cor}
Suppose $u$ is a subsolution in $B_{2}$ with
\[ |\{ x \in B_1; u \leq 0 \}| \geq \epsilon |B_1| \]
and $u \leq 1$. 
Then there exists a constant $0< \mu < 1 $ depending only on $\epsilon, n,$ and $\sigma$ such that 
\[ \sup_{B_\frac{1}{2}} u \leq \mu \]
\end{cor}

In line with \cite{F-K-S}, we let $w(B) = \int_B y^a dx dy$ represent the integral of our weight over a ball. 
The proof of this lemma depends on a Poincare inequality:
\begin{lem}
For any $\epsilon > 0$ there exists a $C(\epsilon,\sigma)$ such that for $u \in H^1(B_1)$ with 
\[ |\{x\in B_1; u = 0\}| \geq \epsilon |B_1| \]
we have 
\[ \int_{B_1} y^a u^2 dx dy  \leq C \int_{B_1} y^a |\nabla u|^2 dx dy \]
\end{lem}
\begin{proof}
It is a classical result (see, for example, Kinderlehrer and Stampacchia, II.A.15 \cite{K-S}) that for smooth functions $u$ on $B_r$ that vanish on a set of measure at least $\epsilon B_r$, we have 
\[ |u(x)| \leq C \int_{B_r} \frac{|\nabla u(z)|}{|x-z|^{n-1}} dz \]
However, this is precisely the point of departure for Fabes, Kenig, and Serapioni (Theorem 1.2) \cite{F-K-S}, where they prove for functions $u$ satisfying this condition, we have
\[ \int_{B_r} y^a u^2 dx dy \leq C r \int_{B_r} y^a |\nabla u|^2 dx dy \]
which is precisely the result we were looking for. 
\end{proof}

The rest of the proof follows the proof given in \cite{Lin}, which is reasonably short, so we reproduce it here.
\begin{proof}
Assume $u \geq \delta > 0$ - we will see that the final result is insensitive to $\delta$, and so we can let $\delta \rightarrow 0+$ at the end.

Let $v = (\log u)^-$, then $v$  is a subsolution to the equation, bounded by $\log \delta^{-1}$. Then we have (Theorem 2.3.1 in \cite{F-K-S})
\[ \sup_{B_\frac{1}{2}} v \leq C \left( \frac{1}{w(B_1)} \int_{B_1} y^a v^2 dx dy \right)^{\frac{1}{2}} \]
Applying the Poincare inequality, we see that 
\[ \sup_{B_\frac{1}{2}} v \leq C \left( \frac{1}{w(B_1)} \int_{B_1} y^a |\nabla v|^2 dx dy \right)^\frac{1}{2} \]
We set the test function $\phi = \frac{\zeta^2}{u}$ for $\zeta \in C_0^1(B_2)$. Then we obtain
\[ 0 \leq \int y^a \nabla u \cdot \nabla (\frac{\zeta^2}{u}) dy dx = - \int \frac{\zeta^2}{u^2} (\nabla u)^2  + 2 \frac{\zeta}{u} \nabla u \cdot \nabla \zeta dx dy \]
whence we obtain
\[ \int y^a \zeta^2 |\nabla (\log u)|^2 dx dy \leq C \int y^a |\nabla \zeta|^2 dx dy \]
By fixing $\zeta = 1$ on $B_1$ and giving it bounded first derivative, we have
\[ \int_{B_1} y^a |\nabla (\log u)|^2 dx dy \leq C w(B_2) \]
Combining our statements, we find
\[ \sup_{B_\frac{1}{2}} (\log u)^- \leq C \left(\frac{w(B_2)}{w(B_1)} \right)^{\frac{1}{2}} \]
which gives 
\[ \inf_{B_\frac{1}{2}} u \geq e^{-C \left(\frac{w(B_2)}{w(B_1)} \right)^{\frac{1}{2}}} \]
$ \frac{w(B_2)}{w(B_1)} $ is bounded since all $A_p$ weights have a doubling property (see \cite{Stein-2}, V.1.5), and hence 
\[ \inf_{B_\frac{1}{2}} u \geq C \]
\end{proof}
\subsection{The De Giorgi-Nash-Moser Harnack Inequality}\label{ap:DNMH}
\begin{thm}[DeGiorgi-Nash-Moser Interior Harnack Inequality]
 Let $u$ be a non-negative solution in $B_1$ to the equation. Then for $r < 1$, we have 
\[ \sup_{B_r} u \leq c (1-r)^{-p} \inf_{B_r} u \]
where $c,p$ do not depend on $r$ or the center of the ball. 
\end{thm}
\begin{proof}
This fact is a straightforward extension of the standard interior Harnack inequality (proved in \cite{F-K-S}), which simply states that, so long as the equation is satisfied in $B_2$, we have 
\[ \sup_{B_\frac{1}{2}} u \leq C \inf_{B_\frac{1}{2}} u \]
where $C > 1$ is invariant under translation or dilation of the ball. In what follows, we assume $r > \frac{1}{2}$, since the standard inequality proves the result for the case $r \leq \frac{1}{2}$, and that the balls are closed. 

Suppose $\frac{1}{2} > r > \frac{1}{4}$. Consider the collection of balls $B_\frac{1}{2} (x)$, where $x \in \partial B_\frac{1}{2}$. The union of these balls, along with $B_\frac{1}{2}(0)$, is precisely $B_\frac{3}{4}$. For every $x \in \partial B_\frac{1}{2}$, we have 
\[ \sup_{B_\frac{1}{4}(x)} u \leq C u(x) \leq C \sup_{B_\frac{1}{2}} u  \]
Let $x^* \in B_{\frac{1}{2}}$ be such that $B_\frac{1}{4}(x^*)$ is a ball containing $\inf_{B_\frac{3}{4}} u$.  Notice that 
\[ C\sup_{B_\frac{1}{2}} u \leq C^2 u(x^*) \leq C^3 \inf_{B_\frac{3}{4}} u \]
Hence, we have 
\[ \sup_{B_\frac{3}{4}} u \leq C^3 \inf_{B_\frac{3}{4}} u \]

We use the same argument to extend from $B_{1-\frac{1}{2^k}}$ to $B_{1-\frac{1}{2^{k+1}}}$ inductively, and we get
\[ \sup_{B_{1-2^{-k}}} u \leq C^{2k-1} \inf_{B_{1-2^{-k}}} u \] 
and so on, until we reach a the first $k$ such that $1-2^{-k} > r$. At this point, we recognize that $k\approx -\log(1-r)$. Plugging in, we get the desired result. 
\end{proof}

\bibliography{optimal_regularity}{}
\bibliographystyle{amsplain}
\end{document}